\documentclass[10pt]{amsart}

\usepackage{latexsym}
\usepackage{amsmath}
\usepackage{amsfonts}
\usepackage{amssymb}

\usepackage{enumitem}

\newtheorem{theorem}{Theorem}[section]
\newtheorem{lemma}[theorem]{Lemma}
\newtheorem{corollary}[theorem]{Corollary}
\newtheorem{proposition}[theorem]{Proposition}

\newtheorem{sublemma}{}[theorem]

\theoremstyle{definition}

\theoremstyle{remark}

\numberwithin{equation}{section}

\newcommand{\ba}{\backslash}

\usepackage{graphicx}
\DeclareGraphicsExtensions{.pdf,.png,.jpg}

\usepackage[ampersand]{easylist}

\begin{document}

\title[Matroid $N$-connectivity]{A notion of minor-based matroid connectivity}

\author{Zachary Gershkoff}
\address{Mathematics Department\\
Louisiana State University\\
Baton Rouge, Louisiana}
\email{zgersh2@math.lsu.edu}

\author{James Oxley}
\address{Mathematics Department\\
Louisiana State University\\
Baton Rouge, Louisiana}
\email{oxley@math.lsu.edu}

\date{\today}

\begin{abstract}
For a matroid $N$, a matroid $M$ is $N$-connected if every two elements of $M$ are in an $N$-minor together. Thus a matroid is connected if and only if it is $U_{1,2}$-connected. This paper proves that $U_{1,2}$ is the only connected matroid $N$ such that if $M$ is $N$-connected with $|E(M)| > |E(N)|$, then $M \backslash e$ or $M / e$ is $N$-connected for all elements $e$. Moreover, we show that $U_{1,2}$ and $M(\mathcal{W}_2)$ are the only  matroids $N$ such that, whenever a matroid has an $N$-minor using $\{e,f\}$ and an $N$-minor using $\{f,g\}$, it also has an $N$-minor using $\{e,g\}$. Finally, we show that $M$ is $U_{0,1} \oplus U_{1,1}$-connected if and only if every clonal class of $M$ is trivial.
\end{abstract}


\maketitle

\section{Introduction}
\label{Introduction}

Our terminology  follows Oxley \cite{oxl2}. We say that a matroid $M$ {\it uses} an element $e$ or a set $Z$ of elements if $e \in E(M)$ or $Z \subseteq E(M)$. Let $N$ be a matroid. A matroid $M$ with $|E(M)| \geq 2$ is {\it $N$-connected} if, for every pair of distinct elements $e,f$ of $E(M)$, there is a minor of $M$ that is isomorphic to $N$ and uses $\{e,f\}$.

We will assume, unless otherwise stated, that the matroids discussed here have at least two elements. Note that $U_{1,2}$-connectivity coincides with the usual notion of connectivity for matroids. Hence, relying on a well-known inductive property of matroid connectivity \cite{wtt}, we have that if $M$ is $U_{1,2}$-connected, $e \in E(M)$, and $|E(M)| \geq 3$, then $M \ba e$ or $M / e$ is $U_{1,2}$-connected. Our first theorem shows that $U_{1,2}$ is the only connected matroid with this property.


\begin{theorem}\label{minor_induction}
Let $N$ be a matroid. If, for every $N$-connected matroid $M$ with $|E(M)| > |E(N)|$ and, for every $e$ in $E(M)$, at least one of $M \ba e$ or $M / e$ is $N$-connected, then $N$ is isomorphic to one of $U_{1,2}$, $U_{0,2}$, or $U_{2,2}$.
\end{theorem}

One attractive property of matroid connectivity is that elements can be assigned to components. 
We say that a matroid $N$ has the {\it transitivity property} if, for every matroid $M$ and every triple $\{e,f,g\} \subseteq E(M)$, if $e$ is in an $N$-minor with $f$, and $f$ is in an $N$-minor with $g$, then $e$ is in an $N$-minor with $g$.  Let $M(\mathcal{W}_2)$  be the rank-$2$ wheel. In Section \ref{components}, we prove the following result. 

\begin{theorem}\label{transitivity3}
The only   matroids with the transitivity property are $U_{1,2}$ and $M(\mathcal{W}_2)$.
\end{theorem}

On combining the last two theorems, we get the following result, which indicates how special the usual matroid connectivity is.

\begin{corollary}
\label{combine}
Let $N$ be a matroid with the transitivity property such that whenever $M$ is an $N$-connected matroid, $e \in E(M)$, and $|E(M)| > |E(N)|$, at least one of $M\ba e$ and $M/e$ is $N$-connected. Then $N \cong U_{1,2}$.
\end{corollary}

The concept of $N$-connectivity can also convey interesting information when $N$ is disconnected, as the next result indicates.

\begin{theorem}\label{loop and coloop}
A matroid $M$ is $U_{0,1} \oplus U_{1,1}$-connected if and only if every clonal class of $M$ is trivial.
\end{theorem}

The paper is structured as follows. In the next section, we recall Cunningham and Edmonds's decomposition theorem for connected matroids that are not $3$-connected, which is a basic tool in our proofs. Sections \ref{3-connected matroids}, \ref{connected matroids}, and \ref{disconnected} treat the cases of $N$-connected matroids when $N$ is $3$-connected, connected, and disconnected, respectively. In particular, we prove Theorems \ref{minor_induction}, and \ref{transitivity3} in Section  \ref{components} and Theorem~\ref{loop and coloop} in Section~\ref{disconnected}. 
Finally, in Section \ref{larger}, we consider what can be said when every set of three elements occurs in some minor. 
Moss \cite{tm1} showed that $3$-connected matroids can be characterized as those in which every set of four elements is contained in a minor isomorphic to a member of $\{\mathcal{W}^2, \mathcal{W}^3, \mathcal{W}^4, M(\mathcal{W}_3), M(\mathcal{W}_4), Q_6\}$.

\section{Preliminaries}
\label{preliminaries}

The concept of $N$-connectivity is closely related to roundedness, which is exemplified by Bixby's \cite{bix} result that if $e$ is an element of a $2$-connected non-binary matroid $M$, then $M$ has a $U_{2,4}$-minor using $e$. Formally, let $t$ be a positive integer and let $\mathcal{N}$ be a class of matroids. A matroid $M$ {\it has an $\mathcal{N}$-minor} if $M$ has a minor isomorphic to a member of $\mathcal{N}$. Seymour \cite{sey2} defined  $\mathcal{N}$ to be {\it $t$-rounded} if, for every $(t+1)$-connected matroid $M$ with an $\mathcal{N}$-minor and every subset $X$ of $E(M)$ with at most $t$ elements, $M$ has an $\mathcal{N}$-minor using $X$. 
Thus Bixby's result shows that $\{U_{2,4}\}$ is $1$-rounded. Seymour \cite{sey} extended this result as follows.

\begin{theorem}\label{nonbinary 2 elts}
Let $M$ be a $3$-connected matroid having a $U_{2,4}$-minor, and let $e$ and $f$ be  elements of $M$. Then $M$ has a $U_{2,4}$-minor using $\{e,f\}$.
\end{theorem}

The connectivity function $\lambda_M$ of a matroid $M$ is defined for every subset $X$ of $E(M)$ by $\lambda_M(X) = r(X) + r(E(M) - X) - r(M)$; 
equivalently, 
$\lambda_M(X) = r(X) + r^*(X) - |X|.$
For disjoint subsets $A,B$ of $E(M)$, define $\kappa_M(A,B) = \min \{ \lambda_M(X) : A \subseteq X \subseteq E(M) - B\}$.

\begin{lemma}\label{connectivity_function}
If $N$ is a minor of $M$ and $A,B$ are disjoint subsets of $E(N)$, then $\kappa_N(A,B) \leq \kappa_M(A,B)$.
\end{lemma}


Next we give a brief outline of Cunningham and Edmonds's   decomposition  \cite{cthesis}  of matroids that are $2$-connected but not $3$-connected.  More complete details can be found in \cite[Section 8.3]{oxl2}. First recall that when $(X,Y)$ is a $2$-separation of a connected matroid $M$, we can write $M$ as $M_X \oplus_2 M_Y$ where $M_X$ and $M_Y$ have ground sets $X \cup p$ and $Y \cup p$.  A {\it matroid-labeled tree} is a tree $T$ with vertex set $\{ M_1, M_2, \ldots, M_n \}$ such that each $M_i$ is a matroid and, for distinct vertices $M_j$ and $M_k$, the sets $E(M_j)$ and $E(M_k)$ are disjoint if $M_j$ and $M_k$ are non-adjacent, whereas if $M_j$ and $M_k$ are joined by an edge $e$, then $E(M_j) \cap E(M_k) = \{e\}$, and $\{e\}$ is not a separator in either $M_j$ or $M_k$.

When $f$ is an edge of a matroid-labeled tree $T$ joining vertices $M_i$ and $M_j$, if we contract the edge $f$, we obtain a new matroid-labeled tree $T/ f$ by relabeling the composite vertex that results from this contraction as $M_i \oplus_2 M_j$, with every other vertex retaining its original label.

A {\it tree decomposition} of a $2$-connected matroid $M$ is a matroid-labeled tree $T$ such that if $V(T) = \{M_1, M_2, \ldots, M_n \}$ and $E(T) = \{ e_1, e_2, \ldots, e_{n-1} \}$, then

\begin{enumerate}[label=(\roman*)]
\item $E(M) = (E(M_1) \cup E(M_2) \cup \cdots \cup E(M_n)) - \{e_1, e_2, \ldots, e_{n-1} \}$;
\item $|E(M_i)| \geq 3$ for all $i$ unless $|E(M)| < 3$, in which case, $n = 1$ and $M_1 = M$; and
\item the label of the single vertex of $T / \{ e_1, e_2, \ldots, e_{n-1} \}$ is $M$.
\end{enumerate}

We call the members of $\{ e_1, e_2, \ldots, e_{n-1} \}$ {\it basepoints} since each member of this set is the basepoint of a $2$-sum when we construct $M$. Cunningham and Edmonds  (in \cite{cthesis}) proved the following (see also \cite[Theorem~8.3.10]{oxl2}).


\begin{theorem}
\label{cet}
Let $M$ be a $2$-connected matroid. Then $M$ has a tree decomposition $T$ in which every vertex label that is not a circuit or a cocircuit is $3$-connected, and there are no adjacent vertices that are both labeled by circuits or are both labeled by cocircuits. Moreover, $T$ is unique up to relabeling of its edges.
\end{theorem}

The tree decomposition $T$ whose existence is guaranteed by the last theorem is called the {\it canonical tree decomposition} of $M$. Although circuits and cocircuits with at most three elements are $3$-connected matroids, when we refer to a {\it $3$-connected vertex}, we shall mean one with at least four elements. Clearly, for each edge $p$ of $T$, the graph $T\ba p$ has two components. Thus $p$ induces a partition of $V(T)$ and a corresponding partition $(X_p,Y_p)$ of $E(M)$. The latter partition is a 
 $2$-separation of $M$;  we say that it is {\it displayed} by the edge $p$. Moreover, $M= M_{X_p} \oplus_2 M_{Y_p}$  where $M_{X_p}$ and $M_{Y_p}$ have ground sets   $X_p \cup p$ and $Y_p \cup p$, respectively. We shall refer to this 2-sum decomposition of $M$ as having been  {\it induced} by the edge $p$ of $T$.


We shall frequently use the following well-known result, which appears, for example, as \cite[Lemma~2.15]{jt}.


\begin{lemma}\label{path in tree}
Let $M_1$ and $M_2$ label distinct vertices in a tree decomposition $T$ of a connected matroid $M$. Let $P$ be the path in $T$ joining $M_1$ and $M_2$, and let $p_1$ and $p_2$ be the edges of $P$ meeting $M_1$ and $M_2$, respectively. Then $M$ has a minor that uses $(E(M_1) \cup E(M_2)) \cap E(M)$ and is isomorphic to the $2$-sum of $M_1$ and $M_2$, with respect to the basepoints $p_1$ and $p_2$.
\end{lemma}

We will  often use the next result,  another consequence of  Theorem~\ref{cet}.


\begin{lemma}\label{replace_basepoint}
Let $(X,Y)$ be a $2$-separation displayed by an edge $p$ in a $2$-connected matroid $M$. Suppose $y \in Y$. Then $M$ has, as a minor, the matroid $M_X(y)$ that is obtained from $M_X$ be relabeling $p$ by $y$. In particular, let $N$ be a $3$-connected minor of $M$ with $|E(N)| \geq 4$ and $|E(N) \cap Y| \leq 1$. If $|E(N) \cap Y| = 1$, let $y \in E(N) \cap Y$; otherwise let $y$ be an arbitrary element of $Y$. Then $M_X(y)$ has $N$ as a minor.
\end{lemma}

%

Let $T$ be the canonical tree decomposition of a $2$-connected matroid $M$, and let $M_0$ label a vertex of $T$. Let $p_1, p_2, \ldots, p_d$ be the edges of $T$ that meet $M_0$. For each $p_i$, let $(X_i, Y_i)$ be the $2$-separation of $M$ displayed by $p_i$, where $M_0$ is on the $X_i$-side of the $2$-separation. For each $i$, let $y_i \in Y_i$. Then, by repeated application of Lemma~\ref{replace_basepoint}, we deduce that $M$ has, as a minor, the matroid that is obtained from $M_0$ by relabeling $p_i$ by $y_i$ for all $i$ in $\{1,2,\ldots,d\}$. We denote this matroid by $M_0(y_1,y_2,\ldots,y_d)$ and call it a {\it specially relabeled} $M_0$-minor of $M$.

The following result, which is straightforward to prove by repeated application of Lemma~\ref{connectivity_function}, is well known.

\begin{lemma}\label{replace_basepoint2}
Let $N$ be a $3$-connected matroid with $|E(N)| \geq 3$. Let $M$ be a $2$-connected matroid with canonical tree decomposition $T$. Then there is a unique vertex $M'$ of $T$ such that, for each edge $p$ of $T$,   the partition of $V(T)$ induced by $p$ has the vertex $M'$   on the same side as at least $|E(N)| - 1$ elements of $N$. Moreover, there is a specially relabeled $M'$-minor of $M$ that has $N$ as a minor.
\end{lemma}

\section{$3$-connected matroids}
\label{3-connected matroids}

Let $\mathcal{N}$ be a set of matroids. A matroid $M$ is {\it $\mathcal{N}$-connected} if, for every two distinct elements $e$ and $f$ of $M$, there is an $N$-minor of $M$ that uses $\{e,f\}$ for some $N$ in $\mathcal{N}$. A consequence of \cite[Proposition 4.3.6]{oxl2} is that 
a matroid with at least three elements is $\{U_{1,3}, U_{2,3}\}$-connected if and only if it is connected. The first result in this section characterizes $U_{2,3}$-connected matroids. 
One may hope for a characterization of $3$-connectivity in terms of $\mathcal{N}$-connectivity, but no such characterization exists. To see this, note that 
if $M$ is $\mathcal{N}$-connected, then so is $M \oplus_2 M$.  A characterization of $3$-connectivity in terms of minors containing $4$-element sets, as opposed to the $2$-element sets currently under consideration, is given in \cite{tm1}.


\begin{proposition}\label{U23}
A matroid $M$ is $U_{2,3}$-connected if and only if $M$ is connected and simple.
\end{proposition}

\begin{proof}

Suppose $M$ is $U_{2,3}$-connected. Clearly $M$ is connected and simple. Conversely, if $M$ is connected and simple, and $e$ and $f$ are distinct elements of $M$, then $M$ has a circuit $C$ containing $\{e,f\}$ and $|C| \geq 3$. Hence $M$ has a $U_{2,3}$-minor using $\{e,f\}$, so $M$ is $U_{2,3}$-connected.
%
\end{proof}

\begin{corollary}\label{U13}
A matroid $M$ is $U_{1,3}$-connected if and only if $M$ is connected and cosimple.
\end{corollary}

We will describe $N$-connectivity for a $3$-connected matroid $N$ by first considering the case when $N$ is $U_{2,4}$.
We will refer to binary and non-binary matroids that label vertices of a canonical tree decomposition as {\it binary} and {\it non-binary vertices}.

\begin{theorem}
A matroid $M$ is $U_{2,4}$-connected if and only if $M$ is connected and non-binary, and, in the canonical tree decomposition  of $M$,
\begin{enumerate}[label=(\roman*)]
\item every binary vertex has at most one element that is not a basepoint; and
\item on every path between two binary vertices that each contain a unique element of $E(M)$, there is a non-binary vertex.
\end{enumerate}
\end{theorem}

\begin{proof}
Suppose $M$ is non-binary and connected, and the canonical tree decomposition $T$ of $M$ satisfies the above conditions. Suppose $e$ and $f$ are distinct elements of $M$. If $e$ and $f$ are in the same $3$-connected vertex $M_0$ of $T$, then, by (i), $M_0$ is non-binary. Thus, by Theorem~\ref{nonbinary 2 elts}, $M$ has a $U_{2,4}$-minor using $\{e,f\}$. 

Next suppose $e$ belongs to a binary vertex $M_1$ of $T$, and $f$ belongs to a non-binary vertex $M_0$ of degree $d$. By Lemma~\ref{replace_basepoint}, $M$ contains a specially labeled $M_0$-minor $M_0(e, y_2, y_3, \ldots, y_d)$ using $\{e,f\}$. Similarly, let $e$ and $f$ belong to binary vertices $M_1$ and $M_2$, and let $M_0$ be a  non-binary vertex on the path between them in $T$. Then $M$ contains a specially labeled $M_0$-minor $M_0(e,f,y_3,y_4,\ldots,y_d)$.
%
Thus, by Theorem~\ref{nonbinary 2 elts}, $M$ has a $U_{2,4}$-minor  using $\{e,f\}$.

Suppose now that $M$ is $U_{2,4}$-connected. Clearly $M$ is non-binary and connected. 
If a binary vertex   $M_1$ in $T$  contains two non-basepoints $e$ and $f$, then, by Lemma~\ref{replace_basepoint2},  a $U_{2,4}$-minor of $M$ using $\{e,f\}$ must be a minor of $M_1$; a contradiction.

%

Now suppose $e$ and $f$ are the unique non-basepoints of binary vertices $M_1$ and $M_2$, respectively, in $T$, and let $N$ be a $U_{2,4}$-minor of $M$ using $\{e,f\}$. By Lemma~\ref{replace_basepoint2}, $T$ has a nonbinary vertex $M_0$ such that, for  every edge $p$ of $T$, the partition of $V(T)$ induced by $p$ has $M_0$ on the same side as at least $|E(N)| - 1$ elements of $N$. Let $p_1$ be the edge incident with $M_0$ such that $M_1$ and $M_0$ are on opposite sides of the induced partition of $V(T)$. Then $M_2$  must be on the same side of this partition as $M_0$. Hence $M_0$ lies on the path in $T$ between $M_1$ and $M_2$.
\end{proof}

The last theorem can be generalized  as follows.

\begin{theorem}\label{general}
Let $N$ be a $3$-connected matroid with at least four elements. A matroid $M$ is $N$-connected if and only if $M$ is connected, has $N$ as a minor, and, in the canonical tree decomposition  of $M$,
\begin{enumerate}[label=(\roman*)]
\item every vertex that is not $N$-connected has at most one element that is not a basepoint; and
\item on every path between two vertices that are not $N$-connected and that each have   unique non-basepoints, there is an $N$-connected vertex.
\end{enumerate}
\end{theorem}

%
%
%
%
%


\section{Connected matroids}
\label{connected matroids}

In this section, we consider $N$-connected matroids when $N$ is connected but not $3$-connected. 

\begin{theorem}\label{2wheel}

A matroid $M$ is $M(\mathcal{W}_2)$-connected if and only $M$ is connected and non-uniform.
\end{theorem}

\begin{proof}
If $M$ is $M(\mathcal{W}_2)$-connected, then it is clearly both connected and non-uniform. 
To prove the converse, suppose $M$ is connected and non-uniform. We argue by induction that $M$ is $M(\mathcal{W}_2)$-connected. This is immediate if $|E(M)| = 4$, since $M(\mathcal{W}_2)$ is the unique $4$-element connected, non-uniform matroid. Assume it holds for $|E(M)| < n$ and let $|E(M)| = n > 4$. Distinguish two elements $x$ and $y$ of $E(M)$. 


Suppose there is an element $e$ of $E(M) - \{x,y\}$ such that $M / e$ is disconnected. Then $M$ is the parallel connection, with basepoint $e$, of two matroids $M_1$ and $M_2$. Now $M \ba e$ is connected. We may assume that it is uniform; otherwise, by the induction assumption, $M \ba e$ and hence $M$ has an $M(\mathcal{W}_2)$-minor using $\{x,y\}$. Now $r(E(M_1) - e) + r(E(M_2) - e) - r(M \ba e) = 1$.
Suppose each of $|E(M_1) - e|$ and $|E(M_2) - e|$ has at least two elements. Then $M \ba e$ has a $2$-separation. Since $M \ba e$ is uniform, it follows that $M \ba e$ is a circuit or a cocircuit. In the latter case, $M$ is also a cocircuit; a contradiction. If $M \ba e$ is a circuit, then $M$ is the parallel connection of two circuits, and $M$ is easily seen to have an $M(\mathcal{W}_2)$-minor using $\{ x,y \}$.

Now suppose that $|E(M_1) - e| = 1$. Thus $M$ has a circuit, $\{ e,f \}$ say, containing $e$. As $M \ba e$ is uniform but $M$ is not, $r(M) \geq 2$, so $M \ba e$ has a circuit containing $\{f,x,y\}$. It follows that $M$ has an $M(\mathcal{W}_2)$-minor with ground set $\{e,f,x,y\}$.

We may now assume that $M / e$ is connected for all $e$ in $E(M) - \{x,y\}$. Moreover, by replacing $M$ with $M^*$ in the argument above, we may also assume that $M \ba e$ is connected for all such $e$. If $M \ba e$ or $M / e$ is non-uniform, then, by the induction assumption, $M$ has an $M(\mathcal{W}_2)$-minor using $\{x,y\}$. Thus both $M \ba e$ and $M / e$ are uniform. Let $r(M \ba e) = r$. Then every circuit of $M \ba e$ has $r+1$ elements. Since $M$ is not uniform, it has a circuit containing $e$ that has at most $r$ elements. Contracting $e$ from $M$ produces a rank-$(r-1)$ matroid having a circuit with at most $r-1$ elements. Since $M / e$ is uniform, this is a contradiction.
\end{proof}

We omit the straightforward proof of the next result.

\begin{lemma}\label{transitivity}
If $M$, $N$, and $N'$ are matroids such that $M$ is $N$-connected and $N$ is $N'$-connected, then $M$ is $N'$-connected.
\end{lemma}


%

If we wish to describe the class of $N$-connected matroids for a $3$-connected matroid $N$, it suffices to describe the $N$-connected matroids that are $3$-connected and then apply Theorem~\ref{general}. If $N$ is not $3$-connected, the task of describing $N$-connected matroids becomes harder, and we omit any attempt to provide a general theorem for $N$-connectivity in this case. We will instead give characterizations for two specific matroids that are not $3$-connected, namely $U_{1,4}$ and its dual $U_{3.4}$. We will use the following theorem of Oxley \cite{oxl1}.


\begin{theorem}
\label{rank and corank 3}
Let $M$ be a $3$-connected matroid having rank and corank at least three,  and suppose that $\{x,y,z\} \subseteq E(M)$. Then $M$ has a minor isomorphic to one of $U_{3,6}, P_6, Q_6, \mathcal{W}^3$, or $M(K_4)$ that uses $\{x,y,z\}$.
\end{theorem}

\begin{proposition} \label{U14}
A $3$-connected matroid $M$ is $U_{1,4}$-connected if and only if either $M \cong U_{2,n}$ for some $n \geq 5$, or $M$ has rank and corank at least three.
\end{proposition}

\begin{proof}
Clearly if $n \geq 5$, then $U_{2,n}$ is $U_{1,4}$-connected. Now assume that $r(M) \geq 3$ and $r^*(M) \geq 3$. Suppose $\{x,y\} \subseteq E(M)$. Then, by Theorem~\ref{rank and corank 3}, $M$ has an $\mathcal{N}$-minor using $\{x,y\}$ where $\mathcal{N}$ is  $\{U_{3,6}, P_6, Q_6, \mathcal{W}^3, M(K_4) \}$. One easily checks that each member of $\mathcal{N}$ is $U_{1,4}$-connected. Hence, by Lemma~\ref{transitivity}, $M$ is $U_{1,4}$-connected.

To prove the converse, assume that $M$ is $U_{1,4}$-connected. Since $r^*(U_{1,4})=3$, it follows that $r^*(M) \geq 3$. The required result holds if $r(M) \geq 3$. But, since $M$ is $3$-connected and $U_{1,4}$-connected, $r(M) \geq 2$. Moreover, if $r(M) = 2$, then $M \cong U_{2,n}$ for some $n \geq 5$.
\end{proof}

%
%

Duality gives a corresponding result for $U_{3,4}$-connectivity.

\begin{corollary} \label{U34}
A $3$-connected matroid $M$ is $U_{3,4}$-connected if and only if either $M \cong U_{n-2,n}$ where $n \geq 5$, or $M$ has rank and corank at least $3$.
\end{corollary}

Observe that this fails to fully characterize $U_{3,4}$-connectivity for if we let $M = M(K_{2,3})$, then $M$ is $U_{3,4}$-connected but none of the matroids in its canonical tree decomposition is $U_{3,4}$-connected. We can instead describe $U_{3,4}$-connectivity in terms of forbidden configurations of matroids in the canonical tree decomposition.

%
%
%
%


\begin{proposition}\label{U34+}
Suppose $M$ is not $3$-connected. Then $M$ is $U_{3,4}$-connected if and only if $M$ is connected and simple, and, in the canonical tree decomposition $T$ of $M$, there is no vertex of degree at most two   that is labeled by some $U_{2,n}$ such that its only neighbors in $T$ are cocircuits that use elements of $E(M)$.
\end{proposition}

\begin{proof}
Let $T$ be the canonical tree decomposition of $M$. Assume $M$ is $U_{3,4}$-connected. Then, by Lemma~\ref{transitivity},  $M$ is $U_{2,3}$-connected, so $M$ is connected and simple. Suppose that $T$ has a vertex $M_0$ whose degree $d$ is at most two such that $M_0$ is labeled by some $U_{2,n}$ and has its only neighbors $M_1,\dots, M_d$ labeled by cocircuits that use elements of $E(M)$. For each $i$ in $\{1,\dots,d\}$, suppose $f_i \in E(M_i) \cap E(M)$. Then $M$ can be obtained from a copy of $U_{2,n}$ using $\{f_1,\dots,f_d\}$ by, for each $i$,  adjoining some  matroid via parallel connection across the basepoint $f_i$. If $ d = 1$, let $f_2$ be an element of $M_0$ other than $f_1$. Clearly $M$ has no circuit using $\{f_1,f_2\}$ that has more than three elements.


Now assume that $M$ is connected and simple and that $T$ satisfies the specified conditions. Let $\{e,f\}$ be a subset of  $E(M)$ that is not contained in a $U_{3,4}$-minor.  Assume first that $e$ and $f$ belong to the same vertex $M_1$ of $T$. As $M$ is simple, $M_1$ is not a cocircuit. Now $M$ has a specially relabeled $M_1$-minor using $\{e,f\}$. Thus, by Corollary~\ref{U34},  $M_1 \cong U_{2,n}$ for some $n \ge 3$. Let $p$ be an edge of $T$ that meets $M_1$. Consider the $2$-sum $N_1 \oplus_2 N_2$ induced by $p$ where $\{e,f\} \subseteq E(N_1)$. Certainly $N_1$ has a circuit containing $\{e,f,p\}$, and $N_2$ has a circuit of size at least three containing $p$. Thus $M$ has a $U_{3,4}$-minor containing $\{e,f\}$; a contradiction.

We now know that $e$ and $f$ belong to distinct vertices $M_1$ and $M_2$ of $T$. Each edge $p$ of the path $P$ in $T$ joining $M_1$ and $M_2$ induces a 2-sum decomposition of $M$ into two matroids, $N_{1p}$ and $N_{2p}$. Moreover, an element $x_i$ of $E(N_{ip})$ is in a circuit of $N_{ip}$ of size at least three containing $p$ unless $x_i$ is parallel to $p$ in $N_{ip}$. Thus $e$ or $f$ is parallel to $p$ in $N_{1p}$ or $N_{2p}$, respectively. Let the edges of $P$, in order, be $p_1,p_2,\dots,p_k$ where $p_1$ meets $M_1$. We may assume that $e$ is parallel to $p_1$ in $N_{1p_1}$. Then the vertex $M_1$ of $T$ containing $e$ is a cocircuit.

Suppose $k \ge 3$. As no two adjacent vertices of $T$ are cocircuits, neither $e$ nor $f$ is parallel to $p_2$ in $N_{1p_2}$ or $N_{2p_2}$. Hence $M$ has a $U_{3,4}$-minor using $\{e,f\}$. This contradiction implies  that $k \in \{1,2\}$. Suppose $k = 2$. Then  $f$ is parallel to $p_2$ in $N_{2p_2}$. Thus $M_2$ is a cocircuit. Since $M$ has no $U_{3,4}$-minor using $\{e,f\}$, the vertex $M_3$ of $T$ that is adjacent to both $M_1$ and $M_2$ is  isomorphic to some $U_{2,n}$. By assumption, $M_3$ must have another neighbor in $T$ to which it is joined by the edge $q$, say. Then, for the 2-sum decomposition $Q_1 \oplus_2 Q_2$ of $M$ induced by $q$, there is a circuit of $Q_1$ containing $\{e,f,q\}$ and a circuit of $Q_2$ of size at least three containing $q$. Thus $M$ has a $U_{3,4}$-minor using $\{e,f\}$. This contradiction implies that $k = 1$. Then $M = N_{1p_1} \oplus_2 N_{2p_1}$. Thus the specially relabeled minor $N_{2p_1}(e)$ uses $\{e,f\}$. Now the canonical tree decomposition $T'$ of $N_{2p_1}(e)$ can be obtained from the component of $T\ba p_1$ using $N_{2p_1}$ by replacing $M_2$ by $M_2(e)$. As $e$ and $f$ are contained in the same vertex of $T'$, we deduce from the second paragraph that $N_{2p_1}(e)$, and hence $M$, has a $U_{3,4}$-minor using $\{e,f\}$; a contradiction.
\end{proof}

\section{Disconnected matroids}
\label{disconnected}

We now turn our attention to $N$-connectivity where $N$ is disconnected. The following is essentially immediate. 


\begin{proposition}\label{U_nn_basis}
Let $n$ be an integer exceeding one. A matroid $M$ is $U_{n,n}$-connected if and only if $M$ is simple with rank at least $n$.
\end{proposition}


%
Recall that elements $x$ and $y$ of a matroid $M$ are {\it clones} if the bijection on $E(M)$ that interchanges $x$ and $y$ but fixes every other element yields the same matroid.
Next we prove Theorem~\ref{loop and coloop}, showing that a matroid is $U_{0,1} \oplus U_{1,1}$-connected if and only if no element has a clone. The proof will use the well-known fact (see, for example, \cite{bs}) that two elements in a matroid are clones if and only if they are in precisely the same  cyclic flats.

\begin{proof}[Proof of Theorem~\ref{loop and coloop}]

Suppose every clonal class of $M$ is trivial and let $x$ and $y$ be distinct elements of $M$. Then $M$ has a cyclic flat $F$ that contains exactly one of $x$ and $y$, say $x$. In $M / (F - x)$, the element $x$ is a loop but $y$ is not. Thus $M$ has a $U_{0,1} \oplus U_{1,1}$-minor using $\{x,y\}$, so $M$ is $U_{0,1} \oplus U_{1,1}$-connected.

Conversely, assume $M$ is $U_{0,1} \oplus U_{1,1}$-connected, but $M$ has elements $x$ and $y$ that are in the same cyclic flats. Suppose that $M / C \ba D \cong U_{0,1} \oplus U_{1,1}$ and $E(M / C \ba D) = \{x,y\}$. Let $x$ be the loop of $M / C \ba D$. Then $x \in \mathrm{cl}_M(C)$. Thus $y \in \mathrm{cl}_M(C)$, so $y$ is a loop in $M / C \ba D$; a contradiction.
\end{proof}

Recall, for the next result, that an element is {\it free} in a matroid if it is not a coloop and every circuit that contains it is spanning.

\begin{theorem}
A matroid $M$ is $U_{1,2} \oplus U_{1,1}$-connected if and only if $M$ is loopless, has at most one coloop, and has at most one free element.
\end{theorem}

\begin{proof}
Clearly if $M$ is $U_{1,2} \oplus U_{1,1}$-connected, then it obeys the specified conditions. 
Conversely, suppose $M$ is loopless, has at most one coloop, and has at most one free element. 
Let $e$ and $f$ be elements of $M$. Suppose first that $M$ is disconnected. If $e$ and $f$ are in the same component, then they are in a $U_{1,2}$-minor of that component, so $M$ has a $U_{1,2} \oplus U_{1.1}$-minor  using $\{e,f\}$. If $e$ and $f$ are in different components, then one of these components is not a coloop. That component has a $U_{1,2}$-minor using $e$ or $f$. It follows that $M$ has a $U_{1,2} \oplus U_{1,1}$-minor  using $\{e,f\}$.

Now suppose $M$ is connected. Suppose  that $e$ is free in $M$. Then $f$ is in some non-spanning circuit, $C_f$. Choose $g$ in $C_f - f$. Contracting $C_f - \{f,g\}$ and deleting every other element of $M$
yields a $U_{1,2} \oplus U_{1,1}$-minor of $M$ using $\{e,f\}$.

Suppose neither $e$ nor $f$ is free in $M$. If there is a non-spanning circuit $C$ containing $\{e,f\}$, we can find a $U_{1,2} \oplus U_{1,1}$-minor by contracting every element of $C$ except $e$ and $f$, and deleting every other element except for one. Now suppose every circuit containing $\{e,f\}$ is spanning. Since $e$ is not free, there is a non-spanning circuit $C$ containing $e$.
Clearly $f \not\in \mathrm{cl}(C)$ otherwise $M|{\mathrm{cl}(C)}$ is a connected matroid of rank less than $r(M)$ so it contains a circuit containing $\{e,f\}$; a contradiction.
Therefore, after we contract all of $C$ except for $e$ and one other element, we see that $f$ will not be a loop. Thus we can find a $U_{1,2} \oplus U_{1,1}$-minor using $\{e,f\}$.
\end{proof}

\begin{corollary}
A matroid $M$ is $U_{1,2} \oplus U_{0,1}$-connected if and only if $M$ is coloopless and has at most one element that is in every dependent flat.
\end{corollary}

\section{$N$-connectivity as compared to connectivity}
\label{components}



Before proving Theorem~\ref{minor_induction}, we state and prove its converse.

\begin{proposition}
If $N \in \{U_{1,2} , U_{0,2} , U_{2,2} \}$, then, for every $N$-connected matroid $M$ with $|E(M)| \geq 3$ and for every $e$ in $E(M)$, at least one of $M \ba e$ or $M / e$ is $N$-connected.
\end{proposition}

\begin{proof}
The result is immediate if $N \cong U_{1,2}$. By duality, it suffices to deal with the case when $N \cong U_{2,2}$. Suppose $M$ is $U_{2,2}$-connected, and $|E(M)| \geq 3$. By Proposition~\ref{U_nn_basis}, $M$ is simple with rank at least two. Therefore if $M$ is $U_{2,2}$-connected and $r(M) > 2$, we can delete any element $e$ of $M$ and still have an $N$-connected matroid. Observe that if $r(M) = 2$, then $M$ must be connected since it is simple. Therefore $M$ has no coloops, so $r(M \ba e) = 2$ for all $e$ of $E(M)$. Thus $M \ba e$ is $U_{2,2}$-connected.
\end{proof}


\stepcounter{theorem}
\begin{proof}[Proof of Theorem~\ref{minor_induction}]

First we consider the case when $N$ is connected. Then $N$ is $U_{1,2}$-connected. Thus, by Lemma~\ref{transitivity}, every $N$-connected matroid is $U_{1,2}$-connected and so is connected. Suppose $M$ is an $N$-connected matroid with $|E(M)| > |E(N)|$. 

Assume $N$ is simple. Then, by Proposition~\ref{U23} and Lemma~\ref{transitivity}, $N$, and hence $M$, is $U_{2,3}$-connected. 
Let $M_1$ and $M_2$ be isomorphic copies of $M$ with disjoint ground sets. Pick arbitrary elements $g_1$ and $g_2$ in $M_1$ and $M_2$, and let $M_3$ be the parallel connection of $M_1$ and $M_2$ with respect to the basepoints $g_1$ and $g_2$, which we relabel as $g$ in $M_3$. Then one easily sees that $M_3$ is $N$-connected.
Let $e, f \in E(M_1) - g$. By assumption, we can remove all the elements of $E(M_1) - \{e,f,g\}$ from $M_3$ via deletion or contraction to obtain a matroid $M_4$ that is still $N$-connected. Since $M_4$ is $U_{2,3}$-connnected, it follows that $\{e,f,g\}$ is a triangle in $M_4$. Moreover,   $\{e,f\}$ is a series pair in $M_4$. However, neither $M_4 \ba e$ nor $M_4 / e$ is $U_{2,3}$-connected since $M_4 \ba e$ is disconnected, and $M_4 / e$ has $f$ and $g$ in parallel.
We deduce that $N$ is not simple. Dually, $N$ is not cosimple. The only uniform matroid that is neither simple nor cosimple is $U_{1,2}$, so either $N \cong U_{1,2}$, or $N$ is non-uniform. 

Next we show that $N$ cannot be non-uniform. Suppose, instead, that $N$ is non-uniform. Then, as $N$ is connected, by Theorem~\ref{2wheel}, $N$ is $M(\mathcal{W}_2)$-connected.

Recall that $M$ is $N$-connected with $|E(M)| > |E(N)|$. Let $n = |E(N)| + 1$ and distinguish elements $e, f$ of $E(M)$. Let each of $M_1, M_2, \ldots, M_n$ be a copy of $M$  and let $e_i$ and $f_i$ be the elements of $M_i$ corresponding to $e$ and $f$. Let $M'$ be the parallel connection of $M_1, M_2, \ldots, M_n$ with respect to the basepoints $e_1,e_2,\ldots,e_n$ where  these elements are relabeled as $e$ in $M'$. By assumption, for each $M_i$, we can remove $E(M_i) - \{e,f_i\}$ from $M'$ in such a way that the resulting matroid $M''$ is $N$-connected. Since $M''$ is connected, it must be isomorphic to $U_{1,n+1}$, which is clearly not $M(\mathcal{W}_2)$-connected; a contradiction.
We conclude that $N$ cannot be non-uniform, and hence the theorem holds when $N$ is connected.

Next we consider the case when $N$ is disconnected, first showing the following.

\begin{sublemma}\label{sublemma}
If each element of $N$ is a loop or a coloop, then $N \cong U_{0,2}$ or $U_{2,2}$.
\end{sublemma}

%
Suppose $n \geq 3$ and let $N \cong U_{n,n}$. Let $M = U_{2,3} \oplus U_{n-2,n-2}$. Then $M$ is $N$-connected, but if $e$ is a coloop of $M$, then neither $M \ba e$ nor $M / e$ has a $U_{n,n}$-minor. Therefore $N \not \cong U_{n,n}$; dually, $N \not \cong U_{0,n}$.

If $N = U_{0,1} \oplus U_{1,1}$, then let $M = M(K_4)$. By Theorem~\ref{loop and coloop}, $M$ is $N$-connected, but, for every $e$ of $E(M)$, both $M \ba e$ and $M / e$ have nontrivial clonal classes and are therefore not $N$-connected. Now assume $N \cong U_{0,n} \oplus U_{m,m}$ for some $n  \geq  2$ and  $m \geq 1$. Then $U_{0,n+1} \oplus U_{m,m}$ is an $N$-connected matroid, say $M$. But if $e$ is a coloop, then neither $M \ba e$ nor $M / e$ has an $N$-minor. On combining this contradiction with duality, we conclude that \ref{sublemma} holds.

Now assume that $N$ has $k+s$ components $N_1, N_2, \ldots, N_{k+ s}$ where those with at least two elements are $N_1, N_2, \ldots, N_{k}$. Then $k \ge 1$. For each $i$ in $\{1,2,\ldots,k\}$, choose an element $e_i$ of $N_i$ and relabel it as $p$. Let $M'$ be the parallel connection of $N_1, N_2, \ldots, N_{k}$ with respect to the basepoint $p$ where we take $M' = N_1$ if $k=1$. Let $N'$ be a copy of $N$ whose ground set is disjoint from $E(N)$, and let $n'_i$ be the component of $N'$ corresponding to $N_i$. 
Let  $M_1 = N' \oplus M'$. 
We   show next  that

\begin{sublemma}
\label{m1n}
$M_1$ is $N$-connected. 
\end{sublemma}

Suppose  $\{e,f\} \subseteq E(M_1)$. Certainly $M_1$ has an $N$-minor using $\{e,f\}$ if   $\{e,f\} \subseteq E(N')$. Next suppose that $e \in E(M')$. Then, since $M'$ is a connected parallel connection, we see that, for each $i$ in $\{1,2,\ldots,k+s\}$, there is an $N_i$-minor of $M'$ using $e$. Thus, if $f \in E(N')$, say $f \in E(N'_j)$, then we can choose $i \neq j$ and get an $N$-minor of $M_1$ using $\{e,f\}$ unless $k= 1 = j$. In the exceptional case, $M'$ has an $N_2$-minor with ground set $\{e\}$ and again we get an $N$-minor of $M_1$ using $\{e,f\}$. We may now assume that $f \in E(M')$, say 
$f \in E(N_j)$. Then $M'$ has an $N_j$-minor using $\{e,f\}$, so $M_1$ has an $N$-minor using $\{e,f\}$. Thus \ref{m1n} holds.

Since $M_1$ is $N$-connected, by assumption, we may delete or contract elements of $M_1$ until we obtain an $N$-connected matroid $M_2$ with $|E(M_2)| = |E(N)| + 1$. In particular, we may remove elements from $M'$ in $M_1$ until a single element  $g$  remains. Now choose $e$ in $E(N'_1)$. Then $M_2 \ba e$ or $M_2/e$ is isomorphic to $N$. But both $M_2 \ba e$ and $M_2/e$ have more one-element components than $N'$; a contradiction.
\end{proof}

Recall that we say that a matroid $N$ has the transitivity property if, for every matroid $M$ and every triple $\{e,f,g\} \subseteq E(M)$, if $e$ is in an $N$-minor with $f$, and $f$ is in an $N$-minor with $g$, then $e$ is in an $N$-minor with $g$. Clearly $N$ has the transitivity property if and only if $N^*$ has the transitivity property.

\begin{lemma}\label{parallel-delete}
Suppose $N$ is a  matroid having the transitivity property. Let $N'$ be     obtained from $N$ by adding an element $f$ in parallel to a non-loop element $e$ of $N$. Then there is an element $g$ of $E(N')$ such that $N' \ba g$ is isomorphic to $N$ and has $\{e,f\}$ as a $2$-circuit. Moreover,   $g$ is in a $2$-circuit in $N$.
\end{lemma}

\begin{proof}
The transitivity property implies that $\{e,f\}$ is in an $N$-minor  of $N'$.  Since $r^*(N') > r^*(N)$,  there must be an element $g$ of $E(N') - \{e,f\}$ such that $N'\ba g \cong N$.  Since we have introduced a new $2$-circuit in constructing $N'$, when we delete $g$, we must destroy a $2$-circuit.
\end{proof}

By the last lemma and duality, we obtain the following result.

\begin{corollary}
\label{pd2}
If $N$ is a matroid having the transitivity property, then $N$ has a component with more than one element.
\end{corollary}



The following elementary observation and its dual will be used repeatedly in the proof of Theorem~\ref{transitivity3}.

\begin{lemma}
\label{transit}
Suppose $N$ is a matroid with the transitivity property. Let $N_0$ be a component of $N$ with the largest number of elements. Suppose $f$ is added in parallel to an element $e$ of $N_0$. Let $N_0'$ and $N'$ be the resulting extensions of $N_0$ and $N$, respectively. Suppose $g \in E(N')$ such that $N'\ba g \cong N$. Then $g \in E(N_0')$.
\end{lemma}

Recall that a set $S$ of elements of a matroid $M$ is a {\it fan} if $|S| \geq 3$ and there is an ordering $(s_1, s_2, \ldots, s_n)$ of the elements of $S$ such that, for all $i$ in $\{1,2, \ldots, n-2\}$,

\begin{enumerate}[label=(\roman*)]
\item $\{s_i, s_{i+1}, s_{i+2}\}$ is a triangle or a triad; and
\item when $\{s_i, s_{i+1}, s_{i+2}\}$ is a triangle, $\{s_{i+1}, s_{i+2}, s_{i+3}\}$ is a triad; and when $\{s_i, s_{i+1}, s_{i+2}\}$ is a triad, $\{s_{i+1}, s_{i+2}, s_{i+3}\}$ is a triangle.
\end{enumerate}


Note that the above extends the definition given in \cite{oxl2} by eliminating the requirement that $M$ be simple and cosimple. We shall follow the familiar practice here of blurring the distinction between a fan and a fan ordering.

\begin{lemma}\label{special_fan}
Let $(s_1, s_2, \ldots, s_n)$ be a fan $X$ in a matroid $M$ such that each of $\{s_1, s_2\}$ and $\{s_{n-1}, s_n\}$ is a circuit or a cocircuit. Then $X$ is a component of $M$.
\end{lemma}

\begin{proof}
By switching to the dual if necessary, we may assume that $\{s_1, s_2, s_3\}$ is a triangle of $M$. Thus $\{s_1, s_2\}$ is a cocircuit. Observe that $\{s_i : i \mathrm{~is~odd} \}$ spans $X$. 
If $n$ is odd, this is immediate, and if $n$ is even, it follows from the fact that $\{s_{n-1}, s_n\}$ is a circuit in this case.
By duality, $\{s_i : i \mathrm{~is~even} \}$ spans $X$ in $M^*$. Hence $r(X) + r^*(X) \leq |X|$; that is, $\lambda(X) \leq 0$, so $X$ is a component of $M$.
\end{proof}

%

We define a {\it special fan} to be a fan $(s_1, s_2, \ldots s_k)$ such that $\{s_1, s_2\}$ is a cocircuit of $M$. We will now show that $U_{1,2}$ and $M(\mathcal{W}_2)$ are the only connected matroids with the transitivity property.

\stepcounter{theorem}
\begin{proof}[Proof of Theorem~\ref{transitivity3}]
 It is clear that $U_{1,2}$ has the transitivity property. By Theorem~\ref{2wheel}, two elements of $M$ are in an $M(\mathcal{W}_2)$-minor together if and only if they are in a connected, non-uniform component together. It follows that $M(\mathcal{W}_2)$ has the transitivity property. 
 
Suppose that $N$ has the transitivity property.    Assume that $N$ is not isomorphic to $U_{1,2}$ or $M(\mathcal{W}_2)$. 
Next we show the following.  

\begin{sublemma}
\label{biggest}
Let $N_0$ be a largest component of $N$. Then $N_0$ is isomorphic to $U_{1,2}$ or $M({\mathcal W}_2)$.
\end{sublemma}

Assume that this assertion fails. Then, by Corollary~\ref{pd2}, $N_0$ has at least two, and hence at least three, elements. Take an element $e$ of $N_0$ and add an element $f$ in series with it. Let the resulting coextensions of $N_0$ and $N$ be $N'_0$ and $N'$, respectively. Then, by the transitivity property, $N'/a \cong N$ for some element $a$ of $E(N') - \{e,f\}$. Furthermore, by the dual of Lemma~\ref{transit}, $a \in E(N_0)$. We deduce that $N_0$ has a 2-cocircuit, say $\{a,b\}$. In $N_0$, add an element $c$ in parallel to $a$ to get $N_1$. Then, by transitivity and Lemma~\ref{transit}, there is an element $s_1$ of $E(N_1) - \{a,c\}$ such that $N_1 \ba s_1 \cong N_0$. Since $N_1\ba b$ has $\{a,c\}$ as a component, the component sizes of $N_1 \ba b$ and $N_0$ do not match, so 
 $s_1 \neq b$. 
Thus $s_1 \in E(N) - \{a,b,c\}$, so $N_1 \ba s_1$ has $\{c,a,b\}$ as a cocircuit. Next add an element $d$ to $N_1 \ba s_1$, putting it in series with $c$. Let the resulting matroid be $N_2$. By the dual of Lemma~\ref{transit}, there is an element $s_2$ of $E(N_2) - \{c,d\}$ such that $N_2/d \cong N_0$. 
Moreover, $s_2$ must be in a $2$-cocircuit of $N_2$, and $s_2$ is in a triangle in $N_2$ as $N_2 /s_2$ must have a $2$-circuit that is not present in $N_2$ since adding $d$ destroyed the 2-circuit $\{a,c\}$. 
Now $s_2 \neq a$ since $N_2 / a$ has $\{c,d\}$ as a component. 

Suppose $s_2 = b$. Then $b$ is in a $2$-cocircuit $\{b,e\}$ in $N_2$. 
Moreover, $N_2$ has a triangle $T$ containing $b$. By orthogonality, $T = \{b,e,a\}$. 
Then $(d,c,a,e)$ is a fan $X$ in $N_2 / b$ having $\{c,d\}$ as a cocircuit and $\{a,e\}$ as a circuit. By Lemma~\ref{special_fan}, $X = E(N_2 / b)$, so $N_0 \cong N_2 / b \cong M(\mathcal{W}_2)$; a contradiction.

We now know that $s_2 \neq b$, so $s_2 \not\in \{a,b,c,d\}$. Thus $N_0$ has  $(d,c,a,b)$ as a special fan. 
Among all the special fans of $N_0$ and $N_0^*$, take one, $(a_1, a_2, \ldots, a_k)$, with the maximum number of elements. Then $k \geq 4$. 
First assume  $\{a_{k-2}, a_{k-1}, a_k\}$ is a triad. Suppose $\{a_{k-1}, a_k\}$ is a $2$-circuit of $N_0$.  Then, by Lemma~\ref{special_fan}, the special fan is the whole component $N_0$.  As $N_0 \not \cong M({\mathcal W}_2)$, we see that $k \ge 6$. 
Add an element $f$ in parallel to   $a_3$ to form a new matroid $N_0'$. Then $\{a_1, a_3\}$ is in an $N_0$-minor of $N_0'$, and so is $\{a_1, f\}$. By the  transitivity property, $N_0'$ has  $\{a_3, f\}$  in an $N_0$-minor. Since $N_0'$ has $\{a_3,f\}$ and $\{a_{k-1},a_k\}$ as its only 2-circuits, while $N_0$ has a single 2-circuit, we deduce that $N_0'\ba a_k \cong N_0$. But every element of $N_0$ is in a cocircuit of size at most three, yet $f$ is in no such cocircuit of $N_0'\ba a_k$; a contradiction. 

It remains to deal with the cases when, in $N_0$, either $\{a_{k-2}, a_{k-1}, a_k\}$ is a triad and $\{a_{k-1}, a_k\}$ is not a circuit, or $\{a_{k-2}, a_{k-1}, a_k\}$ is a triangle. In these cases, add $a_0$ in parallel with $a_1$ to produce $N_3$. To obtain an $N_0$-minor of $N_3$ using $\{a_0, a_1\}$,   we must delete an element $z$ of $N_3$ that belongs to a $2$-circuit. Now $z$ is not in $\{a_2, a_3, \ldots, a_k\}$ as  none of these elements is  in a $2$-circuit, so $N_3 \ba z$ is isomorphic to $N_0$ and has $(a_0, a_1, \ldots, a_k)$ as a special fan. This contradicts our assumption that a special fan    in $N_0$ or $N_0^*$ has at most $k$ elements. We conclude that \ref{biggest} holds.

\begin{sublemma}
\label{noone}
$N$  has no single-element component. 
\end{sublemma}

To see this, let $N_0$ be a largest component of $N$. By \ref{biggest}, $N_0$ is isomorphic to $U_{1,2}$ or $M({\mathcal W}_2)$.  Assume that $N$ has a single-element component $N_1$ with $E(N_1) = \{a\}$. By replacing $N$ by its dual if necessary, we may assume that $a$ is a coloop of $N$. Let $c$ be an element  that is in a 2-cocircuit of $N_0$. Now let $N'$ be obtained from $N$ by adding an element $b$ so that $N'$ has $\{a,b,c\}$ as a triangle and $\{a,b\}$ as a cocircuit. Then, by the transitivity property, $N'\ba g \cong N$ for some element $g$ not in $\{a,b\}$. By the choice of $N_0$, we deduce that $g$ must be in the same component $N_0'$ of $N'$ as $\{a,b,c\}$. Moreover, $g$ must be in a 2-cocircuit of $N'_0$. But $N'_0$ contains no such element. Hence \ref{noone} holds.

\begin{sublemma}
\label{iftwo}
$N$  has a single component of maximum size. 
\end{sublemma}

Assume that this fails, and let $N_0$ and $N_1$ be components of $N$ of maximum size. 
Let $\{a_i,b_i\}$ be a 2-circuit of $N_i$. Let $N'_i$ be obtained from $N_i$ by adding $c_i$ in series with $b_i$. Now take a copy of $U_{2,3}$ with ground set $\{c_0,z,c_1\}$ and adjoin $N'_0$ and $N'_1$ via parallel connection across $c_0$ and $c_1$, respectively. Truncate the resulting matroid to get $N_{01}$. Then $r(N_{01}) = r(N_0) + r(N_1) + 1$. Let $N'$ be obtained from $N$ by replacing  $N_0 \oplus N_1$ by $N_{01}$. 
Now $N_{01} /c_0$ and $N_{01} /c_1$ have $(N_0 \oplus N_1)$-minors using $\{z,c_1\}$ and $\{z,c_0\}$, respectively. Hence $N'/c_0$ and $N'/c_1$ have $N$-minors using $\{z,c_1\}$ and $\{z,c_0\}$. Thus, by transitivity, $N'$ has an $N$-minor $\tilde{N}$ using $\{c_0,c_1\}$. As $r(N') = r(N) + 1$, there are elements $e$, $f$, and $g$ of $E(N') - \{c_0,c_1\}$ such that $\tilde{N} = N' / e \ba f,g$. Now $N'/e$   must have two disjoint 2-circuits  that are not in $N'$. Thus $e \in E(N_{01})$. As $e \not\in \{c_0,c_1\}$, it follows that $N_0 \cong M({\mathcal W}_2) \cong N_1$ and, by symmetry, we may assume that $e = a_0$. But $N_{01}/a_0$ does not have an $(M({\mathcal W}_2) \oplus M({\mathcal W}_2))$-minor. Thus \ref{iftwo} holds.

By \ref{biggest} and \ref{iftwo}, $N$ has a single largest component $N_0$ and it is isomorphic to $M({\mathcal W}_2)$. As $N$ is disconnected, we may assume by duality that $N$ has a component  $N_1$ that is isomorphic to $U_{1,k}$ for some $k$ in $\{2,3\}$.  Now take a copy of $U_{2,3}$ with ground set $\{c_0,z,c_1\}$ and adjoin copies of $U_{2,k+1}$ via parallel connection across $c_0$ and $c_1$, letting the resulting matroid be $N_{01}$. Replacing $N_0 \oplus N_1$ by $N_{01}$ in $N$ to give $N'$, we see that $r(N') = r(N) + 1$. Moreover, $N'/c_0$ and $N'/c_1$ have $N$-minors using $\{c_1,z\}$ and $\{c_0,z\}$, respectively. But  $c_0$ and $c_1$ are the only elements $e$ of $N'$ such that $N'/e$ has    two disjoint 2-circuits  that are not in $N'$. Thus $N'$ has no $N$-minor using $\{c_0,c_1\}$. This contradiction completes the proof of the theorem.
\end{proof}


We conclude this section by proving Corollary~\ref{combine}, which demonstrates how two of the basic properties of matroid connectivity are enough to characterize it. 

\begin{proof}[Proof of Corollary~\ref{combine}]
Assume that $N \not\cong U_{1,2}$. Then, by Theorem~\ref{minor_induction} and duality, we may assume that $N \cong U_{2,2}$. But $U_{2,2}$ does not have the transitivity property as the matroid $U_{1,2} \oplus U_{1,1}$ shows.
\end{proof}

\section{Three-element sets}
\label{larger}

The notion of $N$-connectivity defined here relies on sets of two elements. Sets of size three have already been an object of some study.
Seymour asked whether every $3$-element set in a $4$-connected non-binary matroid belongs to a $U_{2,4}$-minor but Kahn~\cite{kahn} and Coullard~\cite{cou} answered this question negatively. Seymour \cite{ps2} characterized the internally $4$-connected binary matroids that are $U_{2,3}$-connected, but the problem of completely characterizing  when every triple of elements in an internally $4$-connected   matroid is in a $U_{2,3}$-minor remains  open  \cite[Problem 15.9.7]{oxl2}.

For a  $3$-connected binary matroid $M$ having rank and corank at least three, Theorem~\ref{rank and corank 3} shows  that every triple of elements of  $M$ is in an $M(K_4)$-minor. The next result extends this theorem to connected binary matroids. As the proof, which is based on Lemma~\ref{replace_basepoint2}, is so similar to those appearing earlier, we omit the details.

\begin{proposition}
Let $M$ be a connected binary matroid. For every triple $\{x,y,z\} \subseteq E(M)$, there is an $M(K_4)$-minor using $\{x,y,z\}$ if and only if every matroid in the canonical tree decomposition of $M$ has rank and corank at least $3$.
\end{proposition}

\end{document}